\documentclass[11pt,a4paper]{amsart}
\usepackage[latin1]{inputenc}
\usepackage{amsfonts}
\usepackage{amsmath,latexsym,amssymb,amsfonts, amsthm}
\usepackage{cite}
\usepackage{graphicx}
\usepackage{amscd}
\usepackage{color}
\usepackage{bm}             
\usepackage{enumerate}
\usepackage[dvips]{epsfig}
\usepackage{psfrag}
\usepackage{epsfig}

\usepackage{graphicx,color}
\usepackage[colorlinks]{hyperref}
\hypersetup{linkcolor=blue,citecolor=blue,filecolor=black,urlcolor=blue}

\newtheorem{theorem}{Theorem}
\newtheorem{corollary}[theorem]{Corollary}
\newtheorem{proposition}[theorem]{Proposition}
\newtheorem{lemma}[theorem]{Lemma}
\theoremstyle{definition}
\newtheorem{remark}[theorem]{Remark}

\newtheorem{definition}[theorem]{Definition}

\numberwithin{theorem}{section}

\numberwithin{equation}{section}

\newcommand{\Lra}{\Longrightarrow}
\newcommand{\R}{\mathbb{R}}
\newcommand{\N}{\mathbb{N}}

\newcommand{\pa}{\partial}

\renewcommand{\div}{\,{\rm div}\,}

\newcommand{\dist}{{\rm dist}}

\newcommand{\B}{{\bf B}}

\newcommand{\eps}{\varepsilon}

\renewcommand{\phi}{\varphi}

\DeclareMathOperator{\id}{id}

\DeclareMathOperator{\loc}{loc}

\renewcommand{\epsilon}{\varepsilon}

\DeclareMathOperator{\sgn}{{sgn}}

\let\psfragfont\footnotesize
\let\psfragfonta\tiny

\psfrag{Oe}{\psfragfont$\Omega_\eps$}
\psfrag{z0}{\psfragfont$\tilde z_0$}
\psfrag{Ke}{\psfragfont$K_\eps$}
\psfrag{A}{\psfragfont$A_\eps$}
\psfrag{twoeps}{\psfragfont$2\eps$}
\psfrag{0}{\psfragfont$0$}
\psfrag{Y0}{\psfragfont$Y_0$}
\psfrag{rn-1}{\psfragfont$\R^{N-1}$}
\psfrag{rn+}{\psfragfont$\R^N_+$}
\psfrag{or}{\psfragfonta$0$}
\psfrag{u}{\psfragfont$u$}
\psfrag{u?}{\psfragfont$u$ ?}
\psfrag{e1}{\psfragfonta$e_1$}
\psfrag{Al}{\psfragfonta$A_\lambda?$}
\psfrag{-e1}{\psfragfonta$-e_N$}
\psfrag{B}{\psfragfont$\B$}
\psfrag{k}{\psfragfonta$K$}
\psfrag{eh}{\psfragfonta$e_H$}
\psfrag{alh}{\psfragfonta$\alpha_H$}
\psfrag{sigma}{\psfragfonta$\Sigma$}
\psfrag{om}{\psfragfonta$\Omega$}
\psfrag{om1}{\psfragfonta$\Omega_1$}
\psfrag{om2}{\psfragfonta$\Omega_2$}
\psfrag{dh}{\psfragfonta$T_\lambda$}
\psfrag{ph}{\psfragfonta$T_\lambda$}
\psfrag{h}{\psfragfonta$K_\lambda$}
\psfrag{x}{\psfragfont$x$}
\psfrag{z}{\psfragfont$z$}
\psfrag{xl}{\psfragfont$x^\lambda$}
\psfrag{xk}{\psfragfonta$y_k$}
\psfrag{ex}{\psfragfont$e$}
\psfrag{eh}{\psfragfont$e_H$}
\psfrag{y}{\psfragfont$y$}
\psfrag{o}{\psfragfont$0$}
\psfrag{x1}{\psfragfont$\overline x$}
\psfrag{x_1}{\psfragfont$x_1$}
\psfrag{y1}{\psfragfont$\overline y$}
\psfrag{yl}{\psfragfont$y^\lambda$}
\psfrag{lam}{\psfragfont$ $}
\psfrag{b}{\psfragfont$\B$}
\psfrag{pb}{\psfragfont$\partial \B$}
\psfrag{sr}{\psfragfont$S_r$}

\begin{document}

\title[The unique continuation property of sublinear equations]{The unique continuation property of sublinear equations}

\author {Nicola Soave, Tobias Weth}

\date{\today}

\address{Dipartimento di Matematica
Politecnico di Milano, 
Via Edoardo Bonardi 9, 20133 Milano, Italy}

\email{nicola.soave@gmail.com; nicola.soave@polimi.it}

\address{Institut f\"ur Mathematik, Goethe-Universit\"at Frankfurt, D-60629 Frankfurt am Main, Germany}

\email{weth@math.uni-frankfurt.de}



\thanks{\em{Acknowledgements:} Nicola Soave is partially supported through the project ERC grant 2013 n. 339958 ``Complex Patterns for Strongly Interacting Dynamical Systems - COMPAT'', through the project PRIN-2015KB9WPT\texttt{\char`_}010 Grant: ``Variational methods, with applications to problems in mathematical physics and geometry", and through the INDAM-GNAMPA project ``Aspetti non-locali in fenomeni di segregazione".}

\begin{abstract}
We derive the unique continuation property of a class of semi-linear elliptic equations with non-Lipschitz nonlinearities. The simplest type of equations to which our results apply is given as $-\Delta u = |u|^{\sigma-1} u$ in a domain $\Omega \subset \R^N$, with $0 \le \sigma <1$. Despite the sublinear character of the nonlinear term, we prove that if a solution vanishes in an open subset of $\Omega$, then it vanishes necessarily in the whole $\Omega$. We then extend the result to equations with variable coefficients operators and inhomogeneous right-hand side. 
\end{abstract}

\maketitle

\section{Introduction}
\label{sec:introduction}

The unique continuation principle is an important tool in the theory of linear partial differential equations, and it has been studied extensively in the case of linear Schr{\"o}dinger operators. Let $\Omega \subset \R^N$ and $V \in L^1_{\loc}(\Omega)$. The Schr{\"o}dinger type 
equation 
\begin{equation}
\label{linear-general-laplace}
-\Delta u = V(x) u \qquad \text{in $\Omega$} 
\end{equation}
is said to have the {\em unique continuation property (UCP)} if every solution $u$ on $\Omega$ which vanishes on an open subset of $\Omega$ is identically 
zero. In recent decades, various classes of potential functions $V$ have been shown to give rise to this property. For a detailed account of results of this type, we refer to the surveys \cite{kenig:89,koch.tataru:2001}. The unique continuation property is interesting per se, and it has many important consequences. In particular, it is closely related to the strict monotonicity of Dirichlet eigenvalues of Schr{\"o}dinger operators with respect to domain inclusion \cite{de-Figueiredo.gossez:92} and to the absence of eigenvalues inside the essential spectrum \cite{jerison.kenig:85,koch.tataru:2006}. It also gives rise to important energy estimates and compactness properties in the context of semilinear elliptic boundary boundary value problems with a variational structure, see e.g. \cite{liu.su.weth:2006,heinz:95}. 

We also recall two stronger variants of the unique continuation property for (\ref{linear-general-laplace}), and for this we assume for simplicity that $V \in L^\infty_{\loc}(\Omega)$. First, (\ref{linear-general-laplace}) is said to have the {\em strong unique continuation property (SUCP)} if every weak solution $u \in H^1_{\loc}(\Omega)$ which has a zero of infinite order in $\Omega$ is identically zero. Here $x_0 \in \Omega$ is called a zero of infinite order if
$$
\int_{B_r(x_0)}u^2\,dx = O(r^n) \qquad \text{as $r \to 0^+$, for every $n \in \N$.}
$$
Moreover, we say that (\ref{linear-general-laplace}) has the {\em unique continuation property in measure (UCPM)}
if every weak solution $u \in H^1_{\loc}(\Omega)$ which vanishes on a set of positive measure is identically zero. 

The above notions of unique continuation can be generalized to the semilinear equation
\begin{equation}
\label{nonlinear-general-laplace}
-\Delta u = f(x,u) \qquad \text{in $\Omega$}.
\end{equation}
In the case where the nonlinearity $f$ satisfies $f(\cdot,0)=0$ on $\Omega$ and the function 
\begin{equation}
  \label{eq:definition-tilde-f}
(x,u) \mapsto \tilde f(x,u):=
\left\{
  \begin{aligned}
  &\frac{f(x,u)}{u},&& \qquad u \not = 0,\\
  &0,&& \qquad  u = 0, 
  \end{aligned}
\right.
\end{equation}
is locally bounded in $x$ and $u$, any solution of (\ref{nonlinear-general-laplace}) satisfies (\ref{linear-general-laplace}) with $V \in L^\infty_{\loc}(\Omega)$ given by $V(x)=\tilde f(x,u(x))$. Hence the unique continuation results for (\ref{linear-general-laplace}) carry over immediately in this case. In particular, this is true for Lipschitz nonlinearities $f$ with $f(\cdot,0)=0$. The picture changes drastically in the non-Lipschitz case where the function $\tilde f$ defined in (\ref{eq:definition-tilde-f}) is not locally bounded and could even fail to belong to $L^1_{\loc}(\Omega)$. The simplest class of nonlinearities of this type are sublinear homogeneous nonlinearities of order less than one given by 
\begin{equation}
  \label{eq:sublinear-homog-nonli}
f_q(u)=|u|^{q-2}u, \qquad q \in (1,2), \qquad \quad f_1(u)=\sgn(u),\qquad q=1,
\end{equation}
where $\sgn$ is the sign function. It is well known that already the ODE
$$
u'' = |u|^{q-2}u 
$$
admits, in the case $1< q <2$, nontrivial solutions $u \in C^2(\R)$ violating the unique continuation principle, as they can be chosen of the form $u(t)=  \Bigl( \frac{2q}{(2-q)^2}\Bigr)^{\frac{1}{q-2}} (t-t_0)^{\frac{2}{2-q}}$ for $t >t_0$ and $u(t)=0$ for $t \le t_0$, with $t_0 \in \R$ chosen arbitrarily.  On the other hand, it is very easy to see that unique continuation holds for the sublinear ODE with opposite sign, i.e., for
\begin{equation}
\label{eq:ode-correct-sign}
-u'' = |u|^{q-2}u.
\end{equation}
Indeed, since the quantity $\frac{{u'}^2}{2}+ \frac{|u|^q}{q}$ is constant and nonzero along nontrivial solutions, $u$ may only have simple zeros. 

It is clear that ODE arguments of this type do not apply to the higher dimensional analogue of (\ref{eq:ode-correct-sign}) given by the equation 
\begin{equation}
  \label{eq:sublinear-q-eq}
-\Delta u = f_q(u) \qquad \text{in $\Omega$.}
\end{equation}
The study of (\ref{eq:sublinear-q-eq}) is motivated, in particular, by its close relationship to the 
(sign changing) porous medium equation 
\begin{equation}
  \label{eq:scpme}
w_t -\Delta |w|^{m-1}w=0 \qquad \qquad \text{with $\quad m=\frac{1}{q-1} > 1$.}  
\end{equation}
Indeed, as discussed in \cite[Chapter 4]{vazquez:2007}, a solution $u$ of (\ref{eq:sublinear-q-eq}) gives rise to a solution of (\ref{eq:scpme}) with separate variables via the ansatz
$$
(x,t) \mapsto w(x,t)=\Bigl(\frac{2-q}{q-1}(t-t_0)\Bigr)^{-\frac{q-1}{2-q}}|u(x)|^{q-2}u(x), \qquad \qquad t >t_0,\;x \in  \Omega.
$$ 
We remark that nonnegative solutions of (\ref{eq:sublinear-q-eq}) trivially obey the unique continuation principle, as they either vanish identically or they are positive by the strong maximum principle for superharmonic functions. Moreover, if a sign changing solution $u \in C^2(\Omega)$ of (\ref{eq:sublinear-q-eq}) has the property that the set $\{x \in \Omega\::\: u(x) \not = 0\}$ only has finitely many components which all satisfy the interior sphere condition, then the Hopf boundary point Lemma easily implies that $u^{-1}(0)$ is an $(N-1)$-dimensional hypersurface and therefore has zero $N$-dimensional Lebesgue-measure. In particular, this property is shared by radial solutions of (\ref{eq:sublinear-q-eq}) in bounded radial domains, but it fails to hold for general sign changing solutions of (\ref{eq:sublinear-q-eq}).
  
Part of the movitation for the present paper arises from the recent work \cite{parini.weth:2015} where the unique continuation principle has been studied for a special class of sign changing solutions of (\ref{eq:sublinear-q-eq}).
More precisely, in \cite{parini.weth:2015}, Parini and the second author focus on the class of least energy sign changing solutions of the Neumann problem for (\ref{eq:sublinear-q-eq}) in a bounded Lipschitz domain. Combining variational arguments with a blow up procedure and local perturbation arguments, they prove that the zero set of least energy sign changing solutions has Lebesgue measure zero. However, the variational perturbation technique developed in \cite{parini.weth:2015} does not extend to higher energy sign changing solutions. We note that both under Dirichlet and under Neumann boundary conditions, (\ref{eq:sublinear-q-eq}) is known to admit an infinite sequence of sign changing solutions which converges to the trivial solution, see \cite{du:preprint,bartsch.willem:95}. For these solutions, the unique continuation property is unknown up to now. 

One of the main aims of the present paper is to establish the unique continuation property of (\ref{eq:sublinear-q-eq}) without additional assumptions on the solutions. Our first main result is the following.

\begin{theorem}\label{thm: weak unique}
Let $q \in [1,2)$. If a weak solution $u \in H^1_{\loc}(\Omega)$ of (\ref{eq:sublinear-q-eq}) vanishes in a neighbourhood of a point $x_0 \in \Omega$, 
then $u \equiv 0$ in $\Omega$.
\end{theorem} 

We remark at this point that one of the key tools to prove the unique continuation property of linear Schr{\"o}dinger equations of the type (\ref{linear-general-laplace}) are Carleman inequalities, which apply more generally to differential inequalities of the form 
$$
|\Delta u(x)| \le |V(x)u(x)|,\qquad x \in \Omega,
$$
see e.g. \cite{jerison.kenig:85}. Since these differential inequalities do not depend on the sign of $V$, Carleman inequalities do not seem to be the right tool to derive the unique continuation property for (\ref{eq:sublinear-q-eq}). We were instead strongly inspired by the papers \cite{GarLin86, GarLin91, GaSVG, Kuk}, where monotonicity properties and other tools related to the theory of free boundary problems are used. In \cite{GarLin86}, the authors study the equation 
\begin{equation}\label{div-eq-homogeneous}
-\div(A(x) \nabla u) = 0 \qquad \text{in $\Omega$},
\end{equation}
where the matrix function $A=(a_{ij}): \Omega \to \R^{N\times N}$ satisfies the following regularity and elliticity assumptions: 
\begin{itemize}
\item[(A1)] $A$ is symmetric with locally Lipschitz coefficients $a_{ij}: \Omega \to \R$, $i,j=1,\dots,N$, and there exists a continuous function $\lambda: \Omega \to (0,1)$ such that
\[
\lambda(x) |\xi|^2 \le \sum_{i,j=1}^N a_{ij}(x) \xi_i \xi_j \le \lambda^{-1}(x) |\xi|^2 \qquad \text{for $\xi \in \R^N$.}
\]
\end{itemize}
In particular, it is proved in \cite{GarLin86} that (\ref{div-eq-homogeneous}) has the strong unique continuation property. The same result is also proved for weak solutions to \eqref{linear-general-laplace}, under appropriate assumptions on the potential $V$. A unified treatment is later given in \cite{GarLin91}, where more general linear equations are considered. In \cite{Kuk}, the author provides bounds of the admissible orders of vanishing of a solution to \eqref{linear-general-laplace} on manifolds. Some of the tools introduced in \cite{Kuk} have later been adapted and used in \cite{GaSVG}, where the lower-dimensional obstacle problem for the operator $-\div (A(x)\nabla)$ is studied, and the optimal regularity of solutions is proved. 

As we shall see, these adaptations will be useful for us in the generalization of Theorem~\ref{thm: weak unique} involving general divergence type operators and a weaker sublinearity property. More precisely, we consider the equation 
\begin{equation}\label{div eq}
-\div(A(x) \nabla u)= V(x)u + f(x,u) \qquad \text{in $\Omega$}
\end{equation}
in a domain $\Omega \subset \R^N$ with a Borel-measurable function $f: \Omega \times \R \to \R$. 
Let $F: \Omega \times \R \to \R,\: F(x,s)= \int_0^s f(x,t)\,dt$, denote the primitive of $f$. In addition to (A1), we assume the following:
\begin{itemize}
\item[(A2)] $V \in L^\infty_{\loc}(\Omega)$.\vspace{0.1cm}
   
\item[(A3)] $f \in L^\infty_{\loc}(\Omega \times \R)$, and there exists $\eps_0, \kappa_1,\kappa_2>0$ and $q<2$ such that
  \begin{enumerate}
  \item[i)] $0 <f(x,s)s \le q F(x,s)$ for a.e. $s \in (-\eps_0, \eps_0) \setminus \{0\}$, $x \in \Omega$,\vspace{0.1cm}

  \item[ii)] the function $F(\cdot,s)$ is of class $C^1$ on $\Omega$ for every $s \in (-\eps_0,\eps_0)$,\vspace{0.1cm}

 \item[iii)] $|\nabla_1 F(x,s)| \le \kappa_1 F(x,s)$ for all $x \in \Omega$, $s \in (-\eps_0,\eps_0)$,\vspace{0.1cm}

 \item[iv)] $F(x,s) \ge \kappa_2$ for all $x \in \Omega$, $s \in \{-\eps_0,\eps_0\}$.\vspace{0.1cm}

   \end{enumerate}
\end{itemize}
Here we have set $\nabla_1 F = \bigl(\frac{\partial F}{\partial x_1},\dots,\frac{\partial F}{\partial x_N}\bigr)$. In the following, we say that $u \in H^1_{\loc}(\Omega) \cap L^\infty_{\loc}(\Omega)$ is a weak solution of (\ref{div eq}) if the function $x \mapsto f(x,u(x))$ is Lebesgue-measurable on $\Omega$ and 
$$
\int_\Omega  \langle A(x)\nabla u(x), \nabla \phi(x) \rangle\,dx = \int_{\Omega} \Bigl(V(x)u(x)  + f(x,u(x))\Bigr)\phi(x)\,dx
$$
for all $\phi \in C_{c}^\infty(\Omega)$.

\begin{theorem}\label{thm: main 2}
Suppose that (A1)--(A3) are satisfied. If a weak solution $u \in H^1_{\loc}(\Omega) \cap L^\infty_{\loc}(\Omega)$ of (\ref{div eq}) vanishes in a neighbourhood of a point $x_0 \in \Omega$, 
then $u \equiv 0$ in $\Omega$.
\end{theorem} 

\begin{remark}
It is easy to see that the class of homogenous sublinear nonlinearities $f_q$ given in (\ref{eq:sublinear-homog-nonli}) satisfies assumption (A3). More generally, Theorem \ref{thm: main 2} applies to weak solutions $u \in H^1_{\loc}(\Omega)$ of the equation 
\[
-\div(A(x)\nabla u) = h(x,s) + \sum_{k=1}^m c_k(x) |s|^{q_k-2} s \qquad \text{in $\Omega$},
\]
where $A$ satisfies $(A1)$, $q_k \in [1,2)$ for $k=1,\dots,m$, $c_1,\dots, c_m \in C^1(\Omega)$ are positive functions such that $\frac{\nabla c_1}{c_1}, \dots, \frac{\nabla c_m}{c_m} \in L^\infty(\Omega,\R^N)$,  and $h$ is (weakly) superlinear in the sense that $\tilde h \in L^\infty_{\loc}(\Omega \times \R)$ for the 
function 
$$
(x,s) \mapsto \tilde h(x,s):=
\left\{
  \begin{aligned}
  &\frac{h(x,s)}{s},&& \qquad s \not = 0,\\
  &0,&& \qquad  s = 0. 
  \end{aligned}
\right.
$$
Indeed, in this case, every weak solution $u \in H^1_{\loc}(\Omega)$ is contained in $C^{1,\alpha}_{\loc}(\Omega)$ for some $\alpha>0$ by elliptic regularity (see e.g. \cite[Theorem 3.13]{han.lin:2011}), and thus assumptions (A1) and (A2) are satisfied with $V(x) = \tilde h(x,u(x))$. \\
We point out that, while we have to impose the positivity of the sublinear term $f$, no sign-assumption is needed on the superlinear one $h$.
\end{remark}

As far as we know, Theorems~\ref{thm: weak unique} and \ref{thm: main 2} are the first unique continuation results for general solutions of sublinear equations. These results will be proved by combining key estimates from \cite{GarLin86, GaSVG} related to the linear operator $-\div (A(x) \nabla)$ with new arguments to deal with the zero order nonlinearity in the equations (\ref{eq:sublinear-q-eq}) and (\ref{div eq}). We stress that the presence of the sublinear term on the right hand side in (\ref{eq:sublinear-q-eq}) and (\ref{div eq}) drastically changes the features of the problem; in particular, the ``almost-monotonicity of the Almgren's frequency", the key tool in proving the strong unique continuation in \cite{GarLin86, GarLin91, Kuk}, is lost.

We point out that it remains open whether (\ref{eq:sublinear-q-eq}) or (\ref{div eq}) give rise to the strong unique continuation property or the unique continuation property in measure. 

The paper is organized as follows. In Section~\ref{sec: proof simple case} we deal with the special equation (\ref{eq:sublinear-q-eq}). Due to the homogeneity of the nonlinearity and the simple form of the linear part 
of (\ref{eq:sublinear-q-eq}), the proof of Theorem~\ref{thm: weak unique} is considerably easier than the proof of Theorem~\ref{thm: main 2}, and it is instructive to elaborate the main ideas and estimates related to the sublinearity of the problem in this special case. In Section~\ref{sec:general-case}, we then give the proof of Theorem~\ref{thm: main 2}.   

\section{The model problem}\label{sec: proof simple case}

This section is devoted to the proof of Theorem~\ref{thm: weak unique}. From now on, we assume that $u \in H^1_{\loc}(\Omega)$ is a weak solution of 
the equation (\ref{eq:sublinear-q-eq}) for some $q \in [1,2)$. In the case $q \in (1,2)$, classical elliptic regularity yields $u \in C^{2,\alpha}_{\loc}(\Omega)$ for some $\alpha>0$. In the case $q=1$, we still have that $u \in W^{2,p}_{\loc}(\Omega)$ for all $p<\infty$ (see e.g. \cite[Theorem 9.11]{GT}), which implies that $u \in C^{1,\alpha}_{\loc}(\Omega)$ for all $\alpha \in (0,1)$. Moreover, $u$ is a strong solution of (\ref{eq:sublinear-q-eq}), and it follows from the definition of 
weak derivatives that $D^2 u \in L^{p}_{\loc}(\Omega, \R^{N \times N})$ is an a.e. symmetric matrix in $\Omega$.

To prove Theorem~\ref{thm: weak unique}, we now assume that $u \equiv 0$ in a neighborhood of a point $x_0 \in \Omega$. By translations, we may suppose that $x_0=0$, which simplifies some expressions in the following. We then need to prove that 
\begin{equation}
  \label{eq:claim-weak-unique}
u \equiv 0 \qquad \text{in $\Omega$.}
\end{equation}
For $0<r< \dist(0,\R^N \setminus \Omega)$, we define 
$B_r:= B_r(0)$, $S_r:= \partial B_r$, 
$$
H(r):=  \int_{S_r} u^2\,d\sigma \qquad \text{and}\qquad
D(r):= \int_{\Omega}\Bigl(|\nabla u|^2 - |u|^q\Bigr)\,dx = \int_{S_r} u u_\nu\,d\sigma,
$$
where, here and in the following, we set $\nu(x)= \frac{x}{|x|}$ for $x \not =0$ and let $u_\nu(x)= \langle \nabla u(x), \frac{x}{|x|}\rangle$ denote the radial derivative. From now on, we often omit the volume element $dx$ and the surface element $d \sigma$ inside the integrals. It is clear that integrals on $B_r$ or on $S_r$ are computed, respectively, with respect to the Lebesgue measure in $\R^N$ or with respect to the $(N-1)$-dimensional Hausdorff measure.
Then we have 
$$
H'(r)=\frac{N-1}{r}H(r) + 2 \int_{S_r}u u_\nu = \frac{N-1}{r}H(r) + 2 D(r).
$$
As a consequence, whenever $H(r) \neq 0$ we have
\begin{equation}
  \label{eq:derivative-1}
\frac{d}{dr}\Bigl(\log \frac{H(r)}{r^{N-1}}\Bigr) = 
\frac{H'(r)}{H(r)} - \frac{N-1}{r}= 2 \frac{D(r)}{H(r)}= 2 \frac{N(r)}{r}
\end{equation}
where $N$ is the (Almgren frequency) function 
$$
r \mapsto N(r):= \frac{rD(r)}{H(r)}.
$$
In \cite{GarLin86}, the authors could show the Almgren frequency associated to linear equations in indeed monotone, up to an exponential factor. The proof of this fact cannot be extended in the present setting for any solution $u$ to (\ref{eq:sublinear-q-eq}), due to the sublinear nature of the problem. 
 
In what follows we need to consider the derivatives of $D$ and of $N$. 
This is the object of the following two statements, which are inspired by the computations in \cite[Section 4]{GarLin86}.

%

\begin{proposition}\label{der D}
For every $r \in (0, \dist(0,\R^N \setminus \Omega))$, we have the identity
\[
D'(r) = \frac{N-2}{r} D(r) - \frac{C_{N,q}}{q r} \int_{B_r}|u|^q\,dx  + \int_{S_r}\left(2 u_\nu^2 + \left(\frac{2-q}{q}\right)|u|^{q}\right) \,d\sigma.
\]
with $C_{N,q}= 2N -(N-2)q>0$. 
\end{proposition}
\begin{proof}
In the case $q>1$, where $u \in C^2(\Omega)$ and $|u|^q \in C^1(\Omega)$, we have 
\begin{equation}
  \label{eq:D-prime-simple}
D'(r)= \int_{S_r}\Bigl( |\nabla u|^2 - |u|^{q}\Bigr),
\end{equation}
whereas, by the symmetry of the matrix $D^2u$, 
\begin{align*}
\int_{S_r} |\nabla u|^2 &=  \frac{1}{r}\int_{S_r} \langle |\nabla u|^2  x, \nu \rangle = \frac{1}{r} \int_{B_r}\div \Bigl(|\nabla u|^2 x\Bigr) \\
& = \frac{N}{r} \int_{B_r} |\nabla u|^2  + 
\frac{2}{r} \int_{B_r} \langle x,  (D^2 u)\,\nabla u \rangle \\
&= \frac{N}{r} \int_{B_r} |\nabla u|^2  + 
\frac{2}{r} \int_{B_r} \langle (D^2 u)\,x ,   \nabla u \rangle  
\\
&= \frac{N}{r} \int_{B_r} |\nabla u|^2  + 
\frac{2}{r} \int_{B_r} \Bigl \langle \nabla \bigl(\langle \nabla u, x \rangle - u\bigr), \nabla u \Bigr \rangle  
 \\
&= \frac{N-2}{r} \int_{B_r} |\nabla u|^2  - 
\frac{2}{r} \int_{B_r} \langle \nabla u, x \rangle  \Delta u 
 + \frac{2}{r}\int_{S_r} \langle \nabla u,  x \rangle  u_\nu  \\
&= \frac{N-2}{r} \int_{B_r} |\nabla u|^2  + 
\frac{2}{r}  \int_{B_r} \langle \nabla u, x \rangle f_q(u)  
 + 2\int_{S_r} u_\nu^2.
\end{align*}
By integration by parts 
$$
\int_{B_r} \langle \nabla u, x \rangle f_q(u) = \frac{1}{q}  \int_{B_r} \langle \nabla |u|^{q}, x \rangle = \frac{r}{q} \int_{S_r} |u|^{q} -\frac{N}{q} \int_{B_r}|u|^q.
$$
Inserting these identities in (\ref{eq:D-prime-simple}), we find that 
\begin{align*}
D'(r) &= \frac{N-2}{r} \int_{B_r} |\nabla u|^2  - \frac{2N}{qr} \int_{B_r}|u|^q + \int_{S_r}\left(2 u_\nu^2 + \left(\frac{2}{q}-1\right)|u|^{q}\right)\\
&= \frac{N-2}{r} D(r) - \frac{C_{N,q}}{q r} \int_{B_r}|u|^q  + \int_{S_r}\left(2 u_\nu^2 + \left(\frac{2}{q}-1\right)|u|^{q}\right), 
\end{align*}
 as claimed. The same computation is also valid in the case $q=1$ but requires extra justification. First, we use the fact that -- as remarked before -- 
$D^2 u \in L^{p}_{\loc}(\Omega, \R^{N \times N})$ is an a.e. symmetric matrix function. Moreover, as observed in \cite[Proposition 2.7]{hofmann.mitrea.taylor:2010}, the Gauss-Green formula 
\begin{equation}
  \label{eq:gauss-green}
\int_{B_r} \div w = \int_{S_r} \langle w, \nu \rangle
\end{equation}
holds for all vector fields $w \in C(\overline{B_r},\R^N)$ with $\div w \in L^1(B_r)$. In the above computation for the case $q=1$, (\ref{eq:gauss-green}) is applied successively to the vector fields 
$$
w_1=|\nabla u|^2 x,\qquad w_2 = \langle \nabla u, x \rangle \nabla u \qquad \text{and}\qquad  w_3= |u| x. 
$$
Here we note  in particular that 
\[
\div w_3 = \sgn(u) \langle \nabla u,x \rangle +  N |u| \qquad \text{a.e. in $\Omega$.}
 \qedhere
 \]
\end{proof}

We now derive a lower bound for the derivative of the Almgren frequency function $N$. 

\begin{proposition}\label{der N}
With $C_{N,q}$ given in Proposition~\ref{der D}, we have 
$$
N'(r) \ge \frac{1}{H(r)}\left[ \frac{r}{q}\left(2-q\right)  \int_{S_r}|u|^q\,d\sigma-\frac{C_{N,q}}{q} \int_{B_r}|u|^q\,dx\right]  
$$
for every $r \in (0,\dist(0,\R^N \setminus \Omega))$ such that $H(r) \neq 0$.
\end{proposition}

\begin{proof}
We observe that, whenever $H(r) \neq 0$, it results
\begin{align*}
N'(r) &= \frac{1}{H(r)} \left[D(r) + rD'(r)  -r \frac{D(r) H'(r)}{H(r)} \right] \\
& =  \frac{1}{H(r)} \left[(2-N) D(r) + rD'(r)  -2r \frac{D^2(r)}{H(r)} \right].
\end{align*}
Hence, by Proposition \ref{der D} 
$$
N'(r) = \frac{1}{H(r)}\biggl[r \int_{S_r}\left(2 u_\nu^2 + \left(\frac{2-q}{q}\right)|u|^{q}\right)  - \frac{C_{N,q}}{q}\int_{B_r}|u|^q  -  \frac{2 r D^2(r)}{H(r)}\biggr],  
$$
and the thesis follows observing that
$$
\int_{S_r} u_\nu^2  - \frac{D^2(r)}{H(r)}=\int_{S_r} u_\nu^2 
- \frac{\Bigl(\int_{S_r} u u_\nu \Bigr)^2}{\int_{S_r} u^2} \ge 0 
$$
by the Cauchy-Schwarz inequality.
\end{proof}

We may now complete the

\begin{proof}[Proof of Eq. (\ref{eq:claim-weak-unique})]
Without loss of generality, we suppose by contradiction that $\Omega=B_1$, $u \not \equiv 0$ in $B_1$, but $u \equiv 0$ in $B_{t}$ for some small $t$. We denote
\[
d(r):=  \frac{1}{q}\int_{B_r}|u|^q\,dx, \quad \text{so that} \quad d'(r) = \frac{1}{q}\int_{S_r} |u|^q\,d\sigma,
\]
and we set
\[
r_0:= \sup \left\{r \ge 0: d(r) = 0\right\}  \in (0,1).
\]
It is clear that $D(r) = d(r) = 0$ for every $r \in (0,r_0]$, and, since 
$d$ is non-decreasing, that $d(r)>0$ for $r>r_0$. 
 
By Proposition \ref{der D}, 
\[
D'(r) \ge \left(\frac{N-2}r  \right) D(r) - \frac{C_{N,q}}r d(r) + (2-q) d'(r),
\]
whence 
\[
\frac{d}{dr}\left(\frac{D(r)}{r^{N-2}}\right) \ge \frac{2-q}{r^{N-2}} d'(r) - \frac{C_{N,q}}{r^{N-1}} d(r)
\]
for every $r \in (0,1)$. Thus, for $r \in (r_0, 1)$
\[
\frac{d}{dr} \left(\frac{D(r)}{r^{N-2}}\right) \ge (2-q) d'(r) - C_1 d(r) \qquad \text{with}\quad C_1:= \frac{C_{N,q}}{r_0^{N-1}}.
\]
Integrating in $(r_0,r)$ with $r \in (r_0,1)$, we obtain
\begin{align*}
\frac{D(r)}{r^{N-2}} &\ge \frac{D(r_0)}{r_0^{N-2}} + (2-q) (d(r)-d(r_0)) - C_1 \int_{r_0}^r d(s)\,ds  \\
& \ge (2-q) d(r) - C_1(r-r_0)d(r),
\end{align*}
where we used the fact that $D(r_0)=d(r_0)=0$ and the monotonicity of $d$. Fixing 
$r_1 \in (r_0,1)$ such that $C_1(r_1-r_0) < \frac{2-q}{2}$, we thus conclude that 
$$
\frac{D(r)}{r^{N-2}} \ge \frac{2-q}{2} d(r) \qquad \text{for $r \in (r_0,r_1)$},
$$
and hence 
\begin{equation}
\label{eq: lower bound D}
D(r) \ge  C_2 d(r)>0 \qquad  \text{for $r \in (r_0,r_1)$, with $C_2:= \frac{2-q}{2} r_0^{N-2}$.} 
\end{equation}
Next we note that, since $u \not \equiv 0$ in $B_{r_1}$, there must exist
$r_2 \in (r_0,r_1)$ such that $H(r_2) \neq 0$. We fix such a radius $r_2$ and 
define 
\[
r_3:= \inf \left\{r \in (0,r_2): H(s)>0 \ \text{for every $s \in (r, r_2)$}\right\}.
\]
Then we have $r_3 \ge r_0$, $H(r_3) = 0$ and $H(r) \not = 0$ for $r \in (r_3,r_2)$. Hence the the Almgren frequency $N(r)$ is well defined for $r \in (r_3,r_2)$, and we can estimate its derivative with the help of Proposition \ref{der N}, which gives that 
\[
N'(r) \ge - C_{N,q} \frac{d(r)}{H(r)} \qquad \text{for $r \in (r_3,r_2].$}
\]
Combining this with \eqref{eq: lower bound D}, we obtain
\[
\frac{N'(r)}{N(r)} \ge -C_{N,q} \frac{d(r)}{r D(r)}  \ge -\frac{C_{N,q}}{r_0 C_2} =: -C_3 \qquad \text{for $r \in (r_3,r_2]$.}
\]
To sum up, we proved that 
\[
N'(r) \ge -C_3 N(r) \qquad \text{for $r \in (r_3,r_2],$}
\]
and by integrating we deduce that
\[
r \mapsto N(r) e^{C_3 r} \quad \text{is non-decreasing in $(r_3,r_2]$}.
\]
In particular, 
\[
N(r) \le N(r_2) e^{C_3 r_2}=: C_4  \qquad \text{for $r \in (r_3,r_2]$}. 
\]
By \eqref{eq:derivative-1}, this implies that 
\[
\frac{d}{dr} \log \left( \frac{H(r)}{r^{N-1}}\right) \le \frac{2C_4}{r} \le 
\frac{2C_4}{r_0} \qquad \text{for $r \in (r_3,r_2]$,}
\]
whereas on the other hand we have that 
$$
\lim_{r \to r_3^+} \log \left( \frac{H(r)}{r^{N-1}}\right)= -\infty
$$
since $H(r_3)=0$ and $r_3>0$. This is a contradiction, and hence (\ref{eq:claim-weak-unique}) is proved. We thus have finished the proof of Theorem~\ref{thm: weak unique}.
\end{proof}

\section{The general case}
\label{sec:general-case}

This section is devoted to the proof of Theorem \ref{thm: main 2}. Hence, in the following we assume that the hypotheses (A1)--(A3) are satisfied with some constants $C,\eps_0>0$ and $q<2$. First, we note that it easily follows from (A3)i) that, for every $x \in \Omega$, the function $s \mapsto F(x,s)/|s|^q$ is non-increasing on $(0,\eps_0)$, and it is non-decreasing on $(-\eps_0,0)$. Combining these facts with (A3)iv), we infer that
\begin{equation}\label{rem: on S}
F(x,s) \ge \frac{\min\{F(x,\eps_0),F(x,-\eps_0)\}}{\eps_0^q} |s|^q \ge \frac{\kappa_2}{\eps_0^q} |s|^q \qquad \text{for $x \in \Omega$, $0<|s|<\eps_0$}.
\end{equation}
Second, we discuss the regularity of a weak solution $u \in H^1_{\loc}(\Omega) \cap L^\infty_{\loc}(\Omega)$ of (\ref{div eq}). Since $f \in L^\infty_{\loc}(\Omega \times \R)$, it follows that $f(\cdot,u(\cdot)) \in L^\infty_{\loc}(\Omega)$. Thus standard elliptic regularity applies and yields, as in the previous section, that $u \in W^{2,p}_{\loc}(\Omega)$ for all $p<\infty$, and that $u$ is a strong solution of (\ref{div eq}).  Moreover, $u \in C^{1,\alpha}_{\loc}(\Omega)$ for all $\alpha \in (0,1)$, and $D^2 u \in L^{p}_{\loc}(\Omega, \R^{N \times N})$ is an a.e. symmetric matrix in $\Omega$. 

We claim that it suffices to prove the following simplified version of Theorem~\ref{thm: main 2}.

\begin{theorem}
\label{main-2simplified}
Suppose that (A1)--(A3) are satisfied, and let $u \in H^1_{\loc}(\Omega) \cap L^\infty_{\loc}(\Omega)$ be a weak solution of (\ref{div eq}). Moreover, assume in addition that 
\begin{equation}
  \label{eq:extra-assumption}
\text{$0 \in \Omega$, $\quad A(0)= \id,\quad$ and $\quad |u|< \eps_0\; $ in $\Omega$.} 
\end{equation}
If $u \equiv 0$ in a neigborhood of $0$, then $u \equiv 0$ on $B_{\delta_0}(0)$, where $\delta_0:= \dist(0,\R^N \setminus \Omega)$.
\end{theorem}

Assuming for the moment that Theorem~\ref{main-2simplified} holds true,  we can quickly complete the

\begin{proof}[Proof of Theorem \ref{thm: main 2}]
By assumption, the open set 
$$
U:= \{x \in \Omega\::\: \text{$u \equiv 0$ in a neighborhood of $x$}\}, 
$$
is nonempty. Since $\Omega$ is connected, the claim 
$u \equiv 0$ in $\Omega$  follows once we have shown that $\partial U \cap \Omega = \varnothing$. 
We suppose by contradiction that there exists a point $x_* \in \partial U \cap \Omega$. By the continuity of $u$, we have $u(x_*)= 0$, and thus $x_*$ is contained in the open set 
\[
\Omega_0 := \{x \in \Omega\::\: |u(x)| < \eps_0\}.
\]
Using the continuity and positivity of the function $x \mapsto \lambda(x)$ considered in assumption (A1), we may then choose $x_0 \in U$ with 
$$
|x_*-x_0|<\frac{\lambda_* \delta_*}{2},\qquad \quad \text{where 
$\quad \delta_* := \dist(x_0, \R^N \setminus \Omega_0)\quad$ and $\quad \lambda_*= \lambda(x_0) \in (0,1)$.}
$$ 
Our aim is to show that 
\begin{equation}
  \label{eq:u-equiv-00}
u \equiv 0 \qquad \text{in $B_{\lambda_* \delta_*}(x_0)$,}  
\end{equation}
since then we have $x_* \in U \cap \partial U$, contradicting the openness of $U$ in $\Omega$. We now change coordinates via the affine map 
$$
T : \R^N \to \R^N, \qquad T(x)= A(x_0)^{1/2} x + x_0
$$
We note that, as a consequence of assumption (A1), 
\begin{equation*}
\sqrt{\lambda_*}|x-y| \le |T(x)-T(y)| \le \frac{1}{\sqrt{\lambda_*}} |x-y| \qquad \text{for $x,y \in \R^N$}
\end{equation*}
Let $\widetilde \Omega:= T^{-1}(\Omega_0)$. Then we have $0 \in \widetilde \Omega$, $T(0)=x_0$ and 
$$
T^{-1}(B_{\lambda_* \delta_*}(x_0)) \subset  B_{\sqrt{\lambda_*} \delta_*}(0) \subset \widetilde \Omega
$$
Hence, to show (\ref{eq:u-equiv-00}), it suffices to show that 
\begin{equation}
  \label{eq:u-equiv-0}
\tilde u \equiv 0 \qquad \text{on $B_{\sqrt{\lambda_*} \delta_*}(0)$,}  
\end{equation}
for the function 
$$
\tilde u \in H^1_{\loc}(\widetilde \Omega),\qquad 
 \tilde u(x):= u(T(x)).
$$
The function $\tilde u$ is a weak solution of 
\begin{equation*}
-\div(\widetilde A(x) \nabla \tilde u) = \widetilde V(x)\tilde u + \tilde f(x,\tilde u) \qquad \text{in $\widetilde \Omega$},
\end{equation*}
with $\widetilde V(x)= V(T(x))$, $\tilde f(x,s):= f(T(x), s)$ and 
$$
\widetilde A(x):= A(x_0)^{1/2} A^{-1}(T(x)) A(x_0)^{1/2}.
$$ 
We note that -- on $\widetilde \Omega$ -- the matrix-valued function 
$\widetilde A(x)$ still satisfies (A1) once $\lambda$ is replaced by the function 
$x \mapsto \tilde \lambda(x):=\lambda_* \lambda(T(x))$. Moreover, the function $\widetilde V$ satisfies (A2), and the function $\tilde f \in L^\infty_{\loc}(\widetilde \Omega \times \R)$ satisfies assumption (A3) with unchanged values of $\eps_0$ and 
$\kappa_2$ after making $\kappa_1>0$ smaller if necessary. 

By construction, all assumptions of Theorem~\ref{main-2simplified} are satisfied with $u$, $A$, $V$, $\Omega$ replaced $\tilde u$, $\tilde A$, $\tilde V$, $\tilde \Omega$.  Since $\sqrt{\lambda_*} \delta_* \le \dist(0, \R^N \setminus \tilde \Omega)$, Theorem~\ref{main-2simplified} yields (\ref{eq:u-equiv-0}), as required.
\end{proof} 

It thus remains to prove Theorem~\ref{main-2simplified}, and the rest of this section will be devoted to this aim. From now on, we fix a weak solution $u \in H^1_{\loc}(\Omega) \cap L^\infty_{\loc}(\Omega)$ of (\ref{div eq}), and we assume that (\ref{eq:extra-assumption}) is satisfied in addition to (A1)--(A3). We fix an arbitrary $\delta_1 \in (0,\delta_0)$. It then clearly suffices to show the implication 
\begin{equation}
  \label{eq:-r-*-eq}
u \equiv 0 \quad \text{in a neighborhood of $0$}\qquad \Lra \qquad 
u \equiv 0 \quad \text{on $B_{\delta_1}(0)$.}  
\end{equation}
We note that, by assumptions (A1)-(A3) and since $u \in L^\infty(B_{\delta_1}(0))$, the functions $V, \lambda$, $\lambda^{-1}$ and $x \mapsto f(x,u(x))$, $x \mapsto F(x,u(x))$ are all bounded in $B_{\delta_1}(0)$. We aim at adapting the strategy used to deal with the simple equation \eqref{eq:sublinear-q-eq}, and to this purpose we introduce analogues of the functions $H$, $D$, and $N$ defined in Section~\ref{sec: proof simple case}. As before we put $B_r:= B_r(0)$, $S_r:=\partial B_r$ for $r \in (0,\delta_0)$, and, following \cite{GaSVG,Kuk}, we define
\begin{equation*}
\begin{split}
&H(r) :=\int_{S_r} u^2 \mu\, d \sigma,  \quad \text{where}  \quad \mu(x):= \frac{\langle A(x) x, x \rangle}{|x|^2} \\
&D_1(r)  :=  \int_{B_r} \! \langle A(x)\nabla u, \nabla u \rangle\,dx, \quad D(r)  :=  D_1(r)- \int_{B_r} \!\bigl(V(x)u^2 + f(x,u) u\bigr)dx,\\
&\text{and}\quad N(r)  := \frac{r D(r)}{H(r)} \qquad \text{(defined whenever $H(r) \neq 0$)}. 
\end{split}
\end{equation*}
Notice that by the divergence theorem 
\begin{equation}\label{D =}
D(r) = \int_{S_r} u \langle A(x) \nabla u, \nu \rangle\,d\sigma.
\end{equation}
In what follows the dependence of the matrix $A$ with respect to $x$ will often be omitted, for the sake of brevity. For the same reason, we will often omit $dx$ and $d \sigma$ in the integrals. We also need the following notation.

\medskip

\begin{definition}\label{rmk: on uniformity}
For $\gamma \in \R$, the symbol $O(r^\gamma)$ always stands for a function on $(0,\delta_1)$ satisfying $|O(r^\gamma)| \le C r^\gamma$ for $r \in (0,\delta_1)$ with a constant $C>0$.
\end{definition}

Following closely some computations performed in a different setting in \cite{GaSVG,Kuk}, we may now derive asymptotic estimates for the derivatives of $H$ and of $D_1$.

\begin{lemma}\label{lem: der H}
We have that 
\[
H'(r)= 2 D(r) + \left( \frac{N-1}r +O(1) \right) H(r)
\]
for any $r \in (0,\delta_1)$.
\end{lemma}

\begin{proof}
On $S_r$ we have $\nu= x/|x| = \nabla |x|$, and hence $\langle A \nabla |x|,\nu \rangle = \mu$. As a consequence, using also the symmetry of $A$, we obtain
\begin{align*}
H(r) & = \int_{S_r} u^2 \langle A \nabla |x|,\nu \rangle =\int_{B_r}  \div( u^2 A \nabla |x|)  \\\
&= 2\int_{B_r} u \langle A  \nabla u, \nabla |x| \rangle + \int_{B_r}  u^2 \div(A \nabla |x|).
\end{align*}
By \eqref{D =}, we infer that
\begin{equation}
  \label{eq:preliminary-lem-der-H}
H'(r) = 2 D(r) + \int_{S_r} u^2 \, \div(A \nabla |x|).
\end{equation}
Moreover, by \cite[Lemma 4.1]{GaSVG}
\begin{equation}  \label{eq:div-expansion}
\div(A \nabla |x|) = \frac{N-1}{r} + O(1) \qquad \text{for $r \in (0,\delta_1)$,}
\end{equation}
and by definition $\lambda \le \mu \le \lambda^{-1}$. Inserting these expansions in (\ref{eq:preliminary-lem-der-H}), we obtain the claim.
\end{proof}

As in \cite{GaSVG}, it is now convenient to introduce  the quantity
\begin{equation*}
Z(x):= \frac{A(x)x}{|\mu(x)|}= \frac{|x| A(x) \nabla |x|}{\mu(x)},
\end{equation*}
observing that $\langle Z,\nu \rangle = r$ on $S_r$. Also, for future convenience, we note that 
\begin{equation}\label{div Z}
\div Z = N + O(r) \qquad \text{for $r \in (0,\delta_1)$}.
\end{equation}
For the proof, we refer to \cite[Lemma A.5]{GaSVG}.

\begin{proposition}
 \label{sec:general-case-der-D-1}
For $r \in (0,\delta_1)$, we have that 
\begin{align}
D_1'(r) = \left(\frac{N-2}r + O(1)\right)D_1(r)  
&+ 2 \int_{S_r}  \frac{\langle A \nabla u,\nu\rangle^2}{\mu} \nonumber\\
&+ \frac2r\int_{B_r} \bigl(V(x)u+ f(x,u)\bigr) \langle Z, \nabla u \rangle. \label{D1-prime-formula}
\end{align}
\end{proposition}

The proof of this proposition is very similar to the proof of \cite[Theorem A.2]{GaSVG}, where a somewhat different setting is considered. For the convenience of the reader, we give the complete proof of Proposition~\ref{sec:general-case-der-D-1} in the Appendix. The following corollary is an immediate consequence of Proposition~\ref{sec:general-case-der-D-1}. 

\begin{corollary}\label{lem: der D}
For every $r \in (0,\delta_1)$, we have that
\begin{align*}
D'(r)& = \left(\frac{N-2}r + O(1)\right) D_1(r) 
+ 2 \int_{S_r}  \frac{\langle A \nabla u,\nu\rangle^2}{\mu} \\
&  + \frac2r\int_{B_r}\bigl(V(x)u+  f(x,u)\bigr) \langle Z, \nabla u \rangle - \int_{S_r}\bigl(V(x)u^2 +  f(x,u) u\bigr).
\end{align*}
\end{corollary}

Combining Lemma~\ref{lem: der H} and Corollary~\ref{lem: der D}, we may now estimate the derivative of $N$.

\begin{proposition}\label{prop: der N}
For any $r \in (0,\delta_1)$ with $H(r) \neq 0$, we have that 
\begin{align}
N'(r) \ge O(1) N(r) &+ \frac1{H(r)}\bigg[  ( N-2 + O(r) ) \int_{B_r}\bigl(V(x)u^2+  f(x,u) u\bigr)\nonumber  \\ 
 &+2 \int_{B_r}\bigl(V(x)u+ f(x,u)\bigr) \langle \nabla u, Z \rangle   - r\int_{S_r}\bigl(V(x)u^2 + f(x,u) u\bigr)\bigg].\label{1334}
\end{align}
\end{proposition}

\begin{proof}
By Lemma~\ref{lem: der H} and Corollary~\ref{lem: der D}, we compute that 
\begin{align*}
N'(r) =& \frac{1}{H(r)} \Bigl[D(r) + rD'(r)  -r \frac{D(r) H'(r)}{H(r)} \Bigr]\\
=&  \frac{1}{H(r)} \left[(2-N + O(r)) D(r) + rD'(r)  -2r \frac{D^2(r)}{H(r)} \right] \\
 =& \frac{1}{H(r)} \Bigl[ O(r) D(r) +  \left( N-2 + O(r) \right) \int_{B_r}\bigl(V(x)u^2+  f(x,u) u\bigr)  \\ 
&+ 2 \int_{B_r} \bigl(V(x)u +  f(x,u)\bigr) \langle \nabla u, Z \rangle   - r \int_{S_r}\bigl(V(x)u^2 + f(x,u) u\bigr) \Bigr]\\
&+ \frac{2r}{H(r)} \biggl[ \int_{S_r} \frac{\langle A \nabla u, \nu \rangle^2}{\mu} - \frac{\bigl( \int_{S_r} u \langle A \nabla u,\nu \rangle\bigr)^2 }{\int_{S_r} u^2 \mu} \biggr].
\end{align*}
Hence the Cauchy-Schwarz inequality applied to the last term (this is possible by the ellipticity of $A$) yields the desired result.
\end{proof}

\begin{remark}
When $A \equiv  \id$ in $B_{\delta_1}$, in the previous proposition all the error terms $O(1)$ and $O(r)$ vanish, $Z(x) \equiv x$, and in the case 
$V \equiv 0$ and $f(x,u) = f(u) = |u|^{q-2}u$ we obtain exactly Proposition \ref{der N}, since 
$$
\int_{B_r} f(u) \langle \nabla u, Z \rangle  = \int_{B_r}\langle \nabla F(u), Z \rangle = r \int_{S_r} F(u) - N \int_{B_r} F(u)
$$ 
with $F(u)= \int_0^u f(\tau)d\tau= \frac{1}{q}|u|^q$.
\end{remark}

\medskip

To proceed further with the proof of Theorem \ref{thm: main 2}, we introduce the quantity 
$$
d(r):=  \int_{B_r}F(x,u)\,dx \qquad \text{for $r \in (0,\delta_1)$.}
$$
Since $F(x,u(x)) \ge 0$ on $B_{\delta_1}$ by assumptions (A3)i) and (\ref{eq:extra-assumption}), the function $d$ is differentiable with 

$$
d'(r) = \int_{S_r} F(x,u)\,d\sigma \ge 0 \qquad \text{for $r \in (0,\delta_1)$.}
$$

\begin{corollary}
\label{cor-der-N}
For $0<r<\delta_1$, we have that
\begin{equation}
  \label{eq:D-prime-ineq}
  \begin{split}
D'(r) & \ge \left( \frac{N-2}{r} + O(1)\right)D_1(r)  + (2-q)d'(r) \\
& -  \Bigl(\frac{2N}{r}+O(1)\Bigr)d(r)    + \frac2{r} \int_{B_r}   V u \langle \nabla u ,Z\rangle - \int_{S_r} V u^2.
\end{split}
\end{equation}
Moreover, if $H(r) \neq 0$, we have that
\begin{align}
N'(r) \ge O(1) N(r)&+ \frac{1}{H(r)}\biggl[ r (2-q)d'(r) -   \bigl(2N+O(r)\bigr)d(r)  + 2\int_{B_r}  V u \langle \nabla u, Z \rangle\nonumber\\
& +(N-2 + O(r)) \int_{B_r} \Bigl( V u^2 + f(x,u)u \Bigr)  - r\int_{S_r} V u^2 \biggr].  \label{eq:N-prime-ineq}
\end{align}
\end{corollary}

\begin{proof}
Let $r \in (0,\delta_1)$. We first note that
\begin{align}
\int_{B_r}f(x,u) \langle Z, \nabla u \rangle 
&= \int_{B_r} \bigl \langle \nabla_x F(x,u(x))- \nabla_1 F(x,u), Z \bigr \rangle \nonumber\\ 
&=  \int_{S_r}F(x,u) \langle Z,\nu \rangle   - \int_{B_r}\Bigl(F(x,u) \div Z + \langle \nabla_1 F(x,u), Z \rangle \Bigr) \nonumber\\
&=  r\int_{S_r}F(x,u)   - \int_{B_r}\Bigl( (N+ O(r)) F(x,u) + \langle \nabla_1 F(x,u), Z \rangle \Bigr),\label{1332-prelim}
\end{align}
where we recall that $\langle Z, \nu \rangle = r$ on $S_r$, that $\div Z = N +O(r)$ (see \eqref{div Z}), and that $\nabla_1 F(x,s)$ denotes the gradient of $F(x,s)$ with respect to the the first $N$ variables. Moreover, we have 
\begin{equation}
  \label{eq:Z-pointwise-exp}
\|Z\|_{L^\infty(B_r)}= O(r) \qquad \text{for $r \in (0,\delta_1)$}
\end{equation}
by definition of $Z$ and (A1), and thus assumption (A3)iii) gives 
\begin{equation}\label{1332}
\Bigl|\int_{B_r} 
\langle \nabla_1 F(x,u),Z \rangle \Bigr| \le \kappa_1 \int_{B_r} F(x,u) |Z| \le O(r)  \int_{B_r} F(x,u)  \qquad \text{for $r \in (0,\delta_1)$.}
\end{equation}
Combining (\ref{1332-prelim}) and (\ref{1332}) yields that 
\begin{equation}
\label{1332-1}
\int_{B_r}f(x,u) \langle Z, \nabla u \rangle \ge 
r\int_{S_r}F(x,u)   - (N+ O(r))  \int_{B_r} F(x,u).
\end{equation}
Therefore, by Corollary \ref{lem: der D}, we infer that
\begin{align}
D'(r) 
\ge &   \left(\frac{N-2}r + O(1)\right) D_1(r) 
+ 2 \int_{S_r}  \frac{\langle A \nabla u,\nu\rangle^2}{\mu} + \frac2r\int_{B_r}  V u \langle Z, \nabla u \rangle  \nonumber \\
&- \frac2r (N+O(r))\int_{B_r} F(x,u) + \int_{S_r}  \Bigl(2F(x,u) - f(x,u)u-V u^2\Bigr)
\nonumber
\end{align}
for $r \in (0,\delta_1)$. Since also 
\begin{equation}\label{1333}
\int_{S_r}\Bigl(2 F(x,u) - f(x,u)u\Bigr)\,dx \ge (2-q) \int_{S_r} F(x,u)\,dx
\qquad \text{for $0<r<\delta_1$}
\end{equation}
by assumptions (A3)i) and (\ref{eq:extra-assumption}), we obtain \eqref{eq:D-prime-ineq}. In a similar way, estimate \eqref{eq:N-prime-ineq} can be obtained starting from \eqref{1334}, using \eqref{1332-1} and \eqref{1333}.
\end{proof}

We add further basic estimates to simplify the inequalities in Corollary~\ref{cor-der-N}.
  
\begin{lemma}\label{on u^2}
For $r \in (0,\delta_1)$, we have that 
\begin{align}
  \label{eq:u^2-S-r}
\int_{S_r} u^2 \, dx &\le O(1) \|u\|_{L^\infty(S_r)}^{2-q}d'(r)\\
\int_{B_r} u^2 \, dx &\le O(1) \|u\|_{L^\infty(B_r)}^{2-q}d(r) 
  \label{eq:u^2}
\end{align}
and 
\begin{equation}
\label{estimate linear part}
\Bigl| \int_{B_r}  V u \langle \nabla u, Z\rangle \Bigr| \le O(1)\Bigl( d(r)+ D_1(r) \Bigr)
\end{equation}
\end{lemma}

\begin{proof}
For $r \in (0,\delta_1)$, we have, by \eqref{rem: on S}, 
$$
\int_{S_r} u^2  \le \|u\|_{L^\infty(S_r)}^{2-q} \int_{S_r} |u|^q
\le  \frac{\eps_0^q}{\kappa_2}\|u\|_{L^\infty(S_r)}^{2-q} \int_{S_r} F(x,u) = O(1) \|u\|_{L^\infty(S_r)}^{2-q}d'(r), 
$$ 
as claimed in (\ref{eq:u^2-S-r}). Now (\ref{eq:u^2}) follows immediately by integrating (\ref{eq:u^2-S-r}). Moreover, recalling (\ref{eq:Z-pointwise-exp}), we obtain that
\begin{align*}
\Bigl| \int_{B_r}  V u \langle \nabla u, Z\rangle\,dx \Bigr|   &\le 
  O(r) \| V\|_{L^\infty(B_r)} \int_{B_r} |u| |\nabla u| \le O(r^2) \int_{B_r} u^2 + O(1)\int_{B_r} |\nabla u|^2 \\
& \le O(r^2)\|u\|_{L^\infty(B_r)}^{2-q}d(r) + O(1)\int_{B_r} \langle A \nabla u, \nabla u \rangle,
\end{align*} 
where we have used (\ref{eq:u^2}) in the last step. This gives (\ref{estimate linear part}).
\end{proof}

We may now simplify the inequalities in Corollary~\ref{cor-der-N} as follows.

\begin{corollary}
\label{cor-der-N-1}
There exist a constant $C>0$ such that for $0<r<\delta_1$ we have 
\begin{equation}
  \label{eq:D-prime-ineq-1}
D'(r)  \ge - \frac{C}{r} \Bigl[D_1(r)+d(r)\Bigr]   + \Bigl[(2-q)- C \|u\|_{L^\infty(S_r)}^{2-q}\Bigr]d'(r) 
\end{equation}
Moreover, if $H(r) \neq 0$ and $N(r)>0$, we have that 
\begin{align}
N'(r) \ge -C N(r)&+ \frac{1}{H(r)}\biggl[ r \Bigl((2-q)- C \|u\|_{L^\infty(S_r)}^{2-q}\Bigr) d'(r) -  C \Bigl(D_1(r)+ d(r)\Bigr)\biggr] \label{eq:N-prime-ineq-1}
\end{align}
\end{corollary}

\begin{proof}
Inserting successively the estimates (\ref{estimate linear part}) and (\ref{eq:u^2-S-r}) into (\ref{eq:D-prime-ineq}), we obtain that, with a suitable constant $C>0$ changing its value from line to line,  
\begin{align*}
D'(r) & \ge - \frac{C}{r}\Bigl[D_1(r) +d(r)\Bigr] + (2-q)d'(r)  + \frac2{r} \int_{B_r}   V u \langle \nabla u ,Z\rangle - \int_{S_r} V u^2 \\
& \ge - \frac{C}{r}\Bigl[D_1(r)+d(r)\Bigr] + (2-q)d'(r)  -C \int_{S_r}  u^2 \\
& \ge - \frac{C}{r}\Bigl[D_1(r) +d(r)\Bigr] + \Bigl[(2-q)- C\|u\|_{L^\infty(B_r)}^{2-q}\Bigr] d'(r), 
\end{align*}
as claimed in (\ref{eq:D-prime-ineq-1}). Similarly, (\ref{eq:N-prime-ineq-1}) follows by (\ref{eq:N-prime-ineq}), using (\ref{eq:u^2-S-r}), (\ref{estimate linear part}), and the fact that
\[
(N-2+O(r)) \int_{B_r} f(x,u) u \ge -C d(r)
\]
by assumption (A1)i).
\end{proof}

We are now ready to complete the 

\begin{proof}[Proof of Theorem \ref{main-2simplified}]
As noted before, it suffices to prove the implication (\ref{eq:-r-*-eq}). 
Arguing by contradiction, we suppose that $u \equiv 0$ in a neighborhood of $0$, but $u \not \equiv 0$ in $B_{\delta_1}$. By assumption (A3)i), this implies that 
$F(\cdot,u(\cdot)) \equiv 0$ in a neighborhood of $0$ and 
$F(\cdot,u(\cdot)) \not \equiv 0$ in $B_{\delta_1}$.
Consequently, setting 
$$
r_0:= \sup \{r >0\::\: d(r)=0\}, 
$$
we have $0<r_0<\delta_1$ and $u \equiv 0$ on $B_{r_0}$. We first claim that there exists $r_1 \in (r_0,\delta_1)$ and $C_3>0$ such that 
\begin{equation}
\label{lem: lower-bound-D-1}
D(r) \ge C_3 d(r)>0\qquad \text{for every $r \in (r_0,r_1)$.}
\end{equation}
To see this, we recall that by \eqref{eq:D-prime-ineq-1} we have  
\begin{equation*}
D'(r)  \ge - \frac{C}{r_0} \Bigl[D_1(r)+d(r)\Bigr]   + \Bigl[(2-q)- C \|u\|_{L^\infty(S_r)}^{2-q}\Bigr]d'(r) \qquad \text{for $r \in (r_0,\delta_1)$,}
\end{equation*}
whereas (\ref{eq:u^2}) and assumption (A3)i) implies that  
\begin{align}
\bigl|D(r)-D_1(r) \bigr| &\le \|V\|_{L^\infty(B_r)} \int_{B_r} u^2 + \int_{B_r} f(x,u) u \nonumber\\
&\le  \Bigl( O(1) \|u\|_{L^\infty(B_r)}^{2-q} + q\Bigr)d(r) \le O(1)d(r) \label{1711}
\end{align}
Consequently, 
\begin{equation}
  \label{eq:D-prime-ineq-2}
D'(r)  \ge - C_0 D(r) - C_1 d(r)   + \Bigl[(2-q)- C \|u\|_{L^\infty(S_r)}^{2-q}\Bigr]d'(r) \quad \text{for $r \in (r_0,\delta_1)$}
\end{equation}
with $C_0 = \frac{C}{r_0}$ and a further constant $C_1>0$. We also recall that $D(r_0)=d(r_0)=0$ since $u \equiv 0$ on $B_{r_0}$. Therefore (\ref{eq:D-prime-ineq-2}) implies that 
\begin{align}
e^{C_0 r} D(r) & =  \int_{r_0}^r  \frac{d}{d\tau} \left( e^{C_0 \tau} D(\tau) \right) \,d\tau \nonumber \\
&\ge \int_{r_0}^r  e^{C_0 \tau} \Bigl( \Bigl[(2-q)- C \|u\|_{L^\infty(S_\tau)}^{2-q}\Bigr]d'(\tau)-C_1 d(\tau)\Bigr)\,d\tau \nonumber\\
&  \ge e^{C_0 r_0}(2-q)d(r) - C e^{C_0 \delta_0} \|u\|_{L^\infty(B_r)}^{2-q} d(r)
-C_1e^{C_0 \delta_0} d(r)(r-r_0) \nonumber \\
&= e^{C_0 r_0}\Bigl[(2-q)- C \|u\|_{L^\infty(B_r)}^{2-q} e^{C_0(\delta_0-r_0)} -C_1(r-r_0) e^{C_0(\delta_0-r_0)} \Bigr]d(r)
\label{D-ge-d-proof}
\end{align}
for $r \in (r_0,\delta_1)$. Here we used again the fact that $d$ is increasing. Since, by the continuity of $u$, 
$$
\|u\|_{L^\infty(B_r)}^{2-q} \to \|u\|_{L^\infty(B_{r_0})}^{2-q} =0 \qquad \text{as $r \to r_0^+$,}
$$
we may now choose $r_1>r_0$ sufficiently close to $r_0$ such that 
\begin{equation}
  \label{eq:u-continuity-est}
\left(C \|u\|_{L^\infty(B_r)}^{2-q} +C_1(r-r_0)\right) e^{C_0(\delta_0-r_0)} < \frac{2-q}{2} \qquad \text{for $r \in (r_0,r_1)$.}  
\end{equation}
Then (\ref{D-ge-d-proof}) gives rise to (\ref{lem: lower-bound-D-1}) with 
$C_3 := e^{C_0 (r_0-\delta_0)}\frac{2-q}{2}>0$ and this choice of $r_1$.

Since $u \not \equiv 0$ on $B_r$ for every $r \in (r_0,\delta_1)$, we may now pick $r_2 \in (r_0,r_1)$ such that $H(r_2) \neq 0$, and we define
\[
r_3:=  \inf \left\{r \in (0,r_2): H(s)>0 \ \text{for every $s \in (r, r_2)$}\right\} \ge r_0.
\]
For $r \in (r_3,r_2]$, we then have $H(r)>0$ and $D(r)>0$ by (\ref{lem: lower-bound-D-1}), and thus $N(r)$ is well defined and positive. Moreover, we may invoke (\ref{eq:N-prime-ineq-1}) for $r \in (r_3,r_2]$, which yields that 
\begin{align*}
\frac{N'(r)}{N(r)}  &\ge -C+ \frac{1}{D(r)}\biggl[  \Bigl((2-q)- C \|u\|_{L^\infty(S_r)}^{2-q}\Bigr) d'(r) -  \frac{C}{r} \Bigl(D_1(r)+ d(r)\Bigr)\biggr]\\
&\ge -C- \frac{C(D_1(r)+d(r))}{r_0 D(r)} \ge - C_4 \qquad \text{for $r \in (r_3,r_2]$} 
\end{align*}
with a constant $C_4>0$ as a consequence of (\ref{lem: lower-bound-D-1}), (\ref{1711}) and (\ref{eq:u-continuity-est}). Integrating this inequality in $(r,r_2)$ with $r \in (r_3,r_2)$, we infer that
\[
N(r) \le N(r_2) e^{C_4 r_2} =: C_5  \qquad \text{for $r_3 < r \le r_2$}.
\]
At this point we proceed as in the model case studied in Section \ref{sec: proof simple case}. By Lemma \ref{lem: der H}, the boundedness of the frequency implies that 
\[
\frac{d}{dr} \log \left( \frac{H(r)}{r^{N-1}}\right) = \frac{2 N(r)}{r} + O(1) \le C_5   \qquad \text{for $r \in (r_3,r_2],$} 
\]
for a constant $C_5>0$, whereas on the other hand we have that 
$$
\lim_{r \to r_3^+} \log \left( \frac{H(r)}{r^{N-1}}\right)= -\infty
$$
since $H(r_3)=0$. This is a contradiction, and hence (\ref{eq:-r-*-eq}) is proved. We thus have finished the proof of Theorem~\ref{main-2simplified}.
\end{proof}

\appendix

\section{}
\label{sec:appendix}

Here we give the proof of Proposition~\ref{sec:general-case-der-D-1}, following closely the estimates in \cite[Appendix A]{GaSVG}. We do this for the convenience of the reader since the setting of \cite{GaSVG} is somewhat different. 
In the following statement and in the rest of this section, we omit to explicitly write the sum over repeated indexes, in order to ease the notation.

\begin{lemma}\label{lem: as A.9}
For any $r \in (0,\delta_1)$, we have that 
\begin{align*}
\int_{B_r}  &\left\langle Z, \nabla \left(\langle A\nabla u, \nabla u \rangle \right)\right\rangle  = \int_{B_r}  \langle Z, \nabla a_{hl} \rangle \pa_h u \, \pa_l u + 2 \int_{S_r}  \langle Z, \nabla u \rangle \langle A \nabla u,\nu \rangle \\
& + 2 \int_{B_r}  \langle Z,\nabla u \rangle \bigl(V(x)u+ f(x,u)\bigr)  - 2 \int_{B_r}  a_{hl} \, \pa_h Z_j \, \pa_j u \, \pa_l u
\end{align*}
\end{lemma}
\begin{proof}
Similarly as in the proof of \cite[Lemma A.9]{GaSVG}, we compute
\begin{equation}\label{27sep1}
\begin{split}
\int_{B_r}  \left\langle Z, \nabla \left(\langle A\nabla u, \nabla u \rangle \right)\right\rangle &= \int_{B_r} \langle Z, \nabla a_{lh} \rangle \pa_h u\, \pa_l u + 2 \int_{B_r} \langle Z, \nabla \pa_l u\rangle a_{hl} \,\pa_h u \\
& = \int_{B_r} \langle Z, \nabla a_{lh} \rangle \pa_h u\, \pa_l u + 2 \int_{B_r} \pa_l \left( \langle Z,\nabla u \rangle a_{hl} \, \pa_h u \right) \\
& - 2 \int_{B_r} \langle Z, \nabla u\rangle \pa_l\left( a_{hl} \, \pa_h u \right) - 2 \int_{B_r} a_{hl}\, \pa_l Z_j \, \pa_j u \, \pa_h u.
\end{split}
\end{equation}
The second integral on the right hand side gives
\begin{equation*}
 \int_{B_r} \pa_l \left( \langle Z,\nabla u \rangle a_{hl} \, \pa_h u \right)  = \int_{S_r}  \langle Z,\nabla u \rangle a_{hl} \, \pa_h u \, \nu_l=  \int_{S_r}  \langle Z,\nabla u \rangle \langle A \nabla u, \nu \rangle,
\end{equation*}
Here we have used the divergence theorem again in its weak form (\ref{eq:gauss-green}) with the vector field $w= \langle Z,\nabla u \rangle A \nabla u$. Moreover, the third integral on the right hand side of (\ref{27sep1}) gives
\begin{equation*}
 \int_{B_r} \langle Z, \nabla u\rangle \pa_l\left( a_{hl} \, \pa_h u \right) =  \int_{B_r} \langle Z, \nabla u\rangle \div( A \nabla u)  = -\int_{B_r}  \langle Z,\nabla u \rangle \bigl(V(x)u+ f(x,u)\bigr)
\end{equation*}
Plugging these identities into \eqref{27sep1}, we obtain the result.
\end{proof}

\begin{lemma}\label{lem: as A.10}
For any $r \in (0,\delta_1)$, we have that
\begin{align*}
r \int_{S_r}&  \langle A \nabla u , \nabla u \rangle  = \int_{B_r} \div Z   \langle A \nabla u, \nabla u\rangle  + \int_{B_r}  \langle Z, \nabla a_{hl} \rangle \pa_h u \, \pa_l u \\
+& 2 \int_{S_r}  \langle Z, \nabla u \rangle \langle A \nabla u,\nu \rangle 
 - 2 \int_{B_r}  a_{hl} \, \pa_h Z_j \, \pa_j u \, \pa_l u + 2\int_{B_r}\bigl(V(x)u+  f(x,u)\bigr) \langle Z, \nabla u \rangle.
\end{align*}
\end{lemma}
\begin{proof}
Applying the weak divergence theorem (\ref{eq:gauss-green}) to the vector field $\langle A \nabla u , \nabla u \rangle Z$, we find that, for $r \in (0,\delta_1)$,
\[
\int_{S_r}   \langle A \nabla u , \nabla u \rangle \langle Z,\nu\rangle = \int_{B_r} \div Z \langle A \nabla u , \nabla u \rangle + \langle Z, \nabla ( \langle A \nabla u , \nabla u \rangle  ) \rangle.
\]
Recalling that $\langle Z,\nu \rangle = r$ on $S_r$, the thesis follows directly from Lemma \ref{lem: as A.9}.
\end{proof}

We may now complete the

\begin{proof}[Proof of Proposition~\ref{sec:general-case-der-D-1}]
Let $r \in (0,\delta_1)$. Due to Lemma \ref{lem: as A.10}, we have that
\begin{align}
D_1'(r) & =  \int_{S_r}  \langle A\nabla u, \nabla u \rangle \nonumber\\
& = \frac1r \int_{B_r} \div Z \langle A \nabla u, \nabla u\rangle + \frac1r \int_{B_r}  \langle Z, \nabla a_{hl} \rangle \pa_h u \, \pa_l u + \frac2r \int_{S_r}  \langle Z, \nabla u \rangle \langle A \nabla u,\nu \rangle \nonumber \\
&  - \frac2r \int_{B_r}  a_{hl} \, \pa_h Z_j \, \pa_j u \, \pa_l u 
 + \frac2r\int_{B_r}  \bigl(V(x)u +f(x,u)\bigr) \langle Z, \nabla u \rangle. \label{eq-appendix-1-1}
\end{align}
We separately consider the integrals on the RHS of this equation. 
Recalling  that $\div Z = N+ O(r)$ by (\ref{eq:div-expansion}), we have that 
\begin{equation}
  \label{eq:first-int}
\int_{B_r} \div Z \langle A \nabla u, \nabla u\rangle = \bigl(N+O(r)\bigr) \int_{B_r} \langle A \nabla u, \nabla u\rangle = \bigl(N+O(r)\bigr)D_1(r).
\end{equation}
Recalling moreover (\ref{eq:Z-pointwise-exp}), we also have that 
\begin{align}
\Bigl|\frac{1}{r}\int_{B_r}  \langle Z, \nabla a_{hl} \rangle \pa_h u \, \pa_l u\Bigr|\le \frac{\|Z\|_{L^\infty(B_r)} \|a_{hl}\|_{L^\infty(B_r)} }{r} 
\int_{B_r} |\pa_h u\, \pa_l u| \le O(1) \int_{B_r} |\nabla u|^2 \nonumber\\
  \label{eq:second-int}
 \le O(1) D_1(r)
\end{align}
For third integral, we compute, using the symmetry of $A$, 
\begin{equation}
  \label{eq:third-int}
\int_{S_r}  \langle Z, \nabla u \rangle \langle A \nabla u,\nu \rangle =  \int_{S_r} \frac{1}{\mu}\langle A x, \nabla u \rangle \langle A \nabla u,\nu \rangle = r \int_{S_r} \frac{\langle A \nabla u, \nu \rangle^2}{\mu}.
\end{equation}
Finally, to treat the fourth integral, we put $B(x)= A(x)-\id = (b_{ij}(x))_{ij} \in \R^{N \times N}$, so that 
$B(0)=0$ by (\ref{eq:extra-assumption}) and 
\begin{equation}
  \label{eq:Or-bij}
\|b_{ij}\|_{L^\infty(B_r)}=O(r).   
\end{equation}
We then proceed as in \cite[p. 739]{GaSVG}. Since 
$Z_j= \frac{a_{jl}x_l}{\mu}$, we have that 
$$
\partial_h Z_j = \frac{a_{jh}}{\mu} + \frac{[\partial_h a_{jl}] x_l}{\mu} - \frac{ \partial_h \mu}{\mu^2}
$$
and therefore 
\begin{align*}
a_{hl} \, \pa_h Z_j \, \pa_j u \, \pa_l u &= \frac{a_{hl} a_{jh}\, \pa_j u \, \pa_l u}{\mu}  + \frac{x_l\,a_{hl}
[\partial_h a_{jl}]  \, \pa_j u \, \pa_l u}{\mu} - \frac{x_l\, a_{hl} \, a_{jl} \, \partial_h \mu\, \pa_j u \, \pa_l u }{\mu^2} 
\end{align*}
Since furthermore
\begin{equation}
  \label{eq:Or-mu}
\|\frac{1}{\mu}-1\|_{L^\infty(B_r)}=O(r) \qquad \text{and}\qquad \|\frac{1}{\mu^2}-1\|_{L^\infty(B_r)}=O(r),   
\end{equation}
we find that 
\begin{align}
\int_{B_r}  a_{hl} \, \pa_h Z_j \, \pa_j u \, \pa_l u  &= (1+ O(1))\int_{B_r}a_{hl} a_{jh}\, \pa_j u \, \pa_l u 
+ (1+ O(1)) \int_{B_r} x_l\,a_{hl}
[\partial_h a_{jl}]  \, \pa_j u \, \pa_l u \nonumber\\
& - (1+O(1)) \int_{B_r}  x_l\, a_{hl} \, a_{jl} \, \partial_h \mu\, \pa_j u \, \pa_l u .    \label{eq:fourth-int-1}  
\end{align}
We then note that 
\begin{equation}
\label{eq:fourth-int-2}
\int_{B_r}a_{hl} a_{jh}\, \pa_j u \, \pa_l u =(1+O(r))\!\int_{B_r}a_{lj}\, \pa_j u \, \pa_l u  =(1+O(r))D_1(r),
\end{equation}
since $a_{hl} a_{jh} = a_{lj} + a_{hl}b_{jh}$ and 
$$
\Bigl|\int_{B_r} a_{hl} b_{jh}\, \pa_j u \, \pa_l u \Bigr| \le \|a_{hl}\|_{L^\infty(B_r)} \|b_{jh}\|_{L^\infty(B_r)} \int_{B_r} |\partial_j u \partial_l u| \le O(r) D_1(r)
$$
by (\ref{eq:Or-bij}). Moreover, 
\begin{equation}
\Bigl|\int_{B_r}x_l [\partial_h a_{jl}]  a_{hl}\, \pa_j u \, \pa_l u\Bigr| \le r \|a_{hl} \,\partial_h a_{jl}\|_{L^\infty(B_r)}  \int_{B_r} |\pa_j u  \pa_l u|
\le O(r)D_1(r) \label{eq:fourth-int-3}
\end{equation}
and 
\begin{equation}
\Bigl|\int_{B_r}x_l a_{hl} \, a_{jl} \, \pa_j u \, \pa_l u \partial_h \mu \Bigr| \le r \|a_{hl}\,a_{jl}\,\partial_h \mu\|_{L^\infty(B_r)} 
\int_{B_r}|\pa_j u \, \pa_l u|\le O(r)D_1(r).  \label{eq:fourth-int-4}
\end{equation}
In the latter estimate, we used that $\partial_h \mu$ is bounded since $\mu$ is Lipschitz on $B_{\delta_1}$ by (A1). 
Inserting (\ref{eq:fourth-int-2}),~(\ref{eq:fourth-int-3}) and (\ref{eq:fourth-int-4}) in (\ref{eq:fourth-int-1}) gives that 
\begin{equation}
\int_{B_r}  a_{hl} \, \pa_h Z_j \, \pa_j u \, \pa_l u  = (1+ O(1))\int_{B_r}\langle A \nabla u, \nabla u \rangle.   \label{eq:fourth-int}  
\end{equation}
Finally, inserting (\ref{eq:first-int}), (\ref{eq:second-int}), (\ref{eq:third-int}) and (\ref{eq:fourth-int}) in (\ref{eq-appendix-1-1}) yields \eqref{D1-prime-formula}, as desired.
\end{proof}

  
\end{document}